\documentclass[11pt,a4paper]{article}
\usepackage[utf8]{inputenc}
\usepackage{amsmath}
\usepackage{amsfonts}
\usepackage{amssymb}
\usepackage{amsthm}
\usepackage{fullpage}
\usepackage{mdwlist}
\usepackage{url}
\usepackage[numbers,sort,longnamesfirst]{natbib}

\newcommand{\ceil}[1]{\lceil{#1}\rceil}
\newcommand{\floor}[1]{\lfloor{#1}\rfloor}

\newtheorem{theorem}{Theorem}
\newtheorem{lemma}{Lemma}[section]
\newtheorem{conjecture}[lemma]{Conjecture}
\theoremstyle{remark}
\newtheorem{claim}{Claim}

\renewcommand{\thefootnote}{\fnsymbol{footnote}}

\title{\bf Cycles of given size in a dense graph}
\author{Daniel~J. Harvey\footnotemark[2] \qquad David~R. Wood\footnotemark[2]}
\begin{document}
\maketitle
\footnotetext[2]{School of Mathematical Sciences, Monash University, Melbourne, Australia. \url{{Daniel.Harvey,David.Wood}@monash.edu}. Supported by the Australian Research Council.}

\renewcommand{\thefootnote}{\arabic{footnote}}

\begin{abstract}
We generalise a result of \citeauthor{corrhaj} and show that every graph with average degree at least $\tfrac{4}{3}kr$ contains $k$ vertex disjoint cycles, each of order at least $r$, as long as $k \geq 6$. This bound is sharp when $r=3$.
\end{abstract}

\section{Introduction} 
A well-known result by \citet{corrhaj} states that every graph $G$ with at least $3k$ vertices and minimum degree $2k$ contains $k$ (vertex) disjoint cycles. Since every graph $G$ contains a subgraph with minimum degree greater than half the average degree of $G$, this implies that every graph with average degree at least $4k-2$ (and as such at least $4k-1 \geq 3k$ vertices) contains $k$ disjoint cycles. This result is asymptotically sharp since the complete bipartite graph $K_{n,2k-1}$ does not contain $k$ disjoint cycles (since each cycle contains at least two vertices from the small part) but has average degree tending to $4k-2$ as $n \rightarrow \infty$.

Many extensions of the result of \citet{corrhaj} have been established. \citet{just} proved a density version of the theorem of \citeauthor{corrhaj}, showing a graph with $n \geq 3k$ vertices and more than $\max\{(2k-1)(n-k), \binom{3k-1}{2} + (n-3k+1) \}$ edges contains $k$ disjoint cycles. \citet{verstra1} proved that for each integer $k$ there exists an $n_k$ such that every graph with minimum degree at least $2k$ and at least $n_k$ vertices contains $k$ disjoint cycles of the same order. \citet{hwang3} proved that given $k \geq 2$ and $d \geq 2k$ and $n \geq 3k$, every graph with order $n$ and minimum degree at least $d$ contains $k$ disjoint cycles containing at least $\min\{2d,n\}$ vertices. \citet{kenandpals} proved that every graph with minimum degree at least $2k$ and a sufficiently large number of vertices contains $k$ disjoint even cycles except in a handful of exceptional cases.

Other extensions find specified $2$-factors of $G$ (where a $2$-factor is a set of disjoint cycles that span $V(G)$), for example \citep{hwang2,aigbra,hwang4,hwang5,gouldsurvey,ayconj,deg2factor,all2factor}, or replace the minimum degree requirement by alternatives like \citeauthor{ore} type degree conditions\footnote{Recall \citet{ore} showed that an $n$-vertex graph $G$ contains a Hamiltonian cycle if $\deg v + \deg w \geq n$ for every pair of non-adjacent vertices $v,w$}, for example \citep{kky,kotaro,hwang3,gouldsurvey,deg2factor,chord,all2factor}. 

In this paper, we consider a different direction, in which a lower bound on the size of the desired cycles is specified. This direction has previously been investigated by \citet{hwang1}, who proved that every graph with minimum degree at least $2k$ and at least $4k$ vertices (where $k \geq 2$) contains $k$ disjoint cycles of order at least four, except in three exceptional cases. Hence, since every graph with average degree at least $4k-1$ contains a subgraph with average degree $4k-1$ (and thus $4k$ vertices) and minimum degree at least $2k$, it follows a graph with average degree at least $4k-1$ contains $k$ disjoint cycles of order at least four, when $k \geq 2$. (The lower bound on the average degree precludes the subgraph from being one of the exceptional cases; we omit the proof of this fact.) \citet{hwang6,hwang7} previously showed that if $G$ is a balanced bipartite graph with average degree at least $(s-1)k+1$ and at least $sk$ vertices in each part (where $s \geq 2$), then $G$ contains $k$ disjoint cycles of length at least $2s$. We prove the following theorem.

\begin{theorem}
\label{theorem:intermsofsubgraphs}
For integers $k \geq 6$ and $r \geq 3$, every graph with average degree at least $\tfrac{4}{3}kr$ contains $k$ disjoint cycles, each containing at least $r$ vertices.
\end{theorem}

If $r=3$, then this gives the above corollary of the theorem of \citeauthor{corrhaj} (albeit with a restriction on $k$ and without the $-2$). However, Theorem~\ref{theorem:intermsofsubgraphs} allows us to instead ensure the existence of larger cycles. Our approach to proving this result relies heavily on the concept of minors.

\section{A Minor Approach}
Recall a graph $H$ is a minor of a graph $G$ if a graph isomorphic to $H$ can be constructed from $G$ by a series of vertex deletions, edge deletions and edge contractions. Let $H$ be the graph consisting of $k$ disjoint cycles of order $r$. If $G$ contains $k$ disjoint cycles, each with at least $r$ vertices, then clearly $G$ contains $H$ as a minor. Alternatively, if $G$ contains $H$ as a minor, then by ``uncontracting" each vertex of $H$ we obtain a subgraph of $G$ consisting of $k$ cycles, each with at least $r$ vertices. Hence the following theorem is equivalent to Theorem~\ref{theorem:intermsofsubgraphs}.

\begin{theorem}
\label{theorem:cycles}
For integers $k \geq 6$ and $r \geq 3$, every graph with average degree at least $\tfrac{4}{3}kr$ contains $H$ as a minor, where $H$ is the graph consisting of $k$ disjoint cycles of order $r$.
\end{theorem}

A well-known result by \citet{fast1} shows that sufficiently large average degree ($\geq 2^{t-2}$) is sufficient to force the existence of a complete minor $K_t$. Much work was done improving this required lower bound on the average degree, until \citet{fast3} and \citet{fast4,fast5} showed that average degree at least $\Theta(t\sqrt{\ln t})$ is sufficient and best possible. \citet{fast6} later determined the asymptotic constant for this bound. Similarly, we may consider what average degree is required to force an arbitrary given graph $H$ as a minor. \citet{myersthom} answered this question when $H$ is dense. Similarly, sparse graphs of the form $K_{s,t}$ (where $s \ll t$) have been well studied. \citet{k2t-myers} showed that every $n$-vertex graph with more than $\tfrac{1}{2}(t+1)(n-1)$ edges contains a $K_{2,t}$ minor, as long as $t$ is sufficiently large. More recently \citet{edk2t} proved this result for all $t$. \citet{edk3t} proved that every $n$-vertex graph with more than $\tfrac{1}{2}(t+3)(n-2)+1$ edges contains a $K_{3,t}$ minor. \citet{k2t-myers} conjectured that the average degree required to force a $K_{s,t}$ minor is linear in $t$; this was proven independently by \citet{edkst} and \citet{edkst-ko}. In a recent paper on sparse graphs $H$, \citet{superfast5} conjectured that average degree $\tfrac{4}{3}t-2$ was sufficient to force an arbitrary $2$-regular $t$-vertex graph $H$ as a minor. Theorem~\ref{theorem:cycles} essentially proves this conjecture when each component of $H$ has the same order (apart from the $-2$). Thus we can interpret our result in both the theme of \citet{corrhaj} and the theme of \citet{fast1}.

Call a graph $G$ \emph{minimal} if every vertex deletion or edge contraction lowers the average degree $d(G)$. Let $\delta(G)$ be the minimum degree of $G$, and $\tau(G)$ the minimum number of common neighbours for any pair of adjacent vertices in $G$. If $G$ is minimal then $\delta(G) > \tfrac{1}{2}d(G)$ and $\tau(G) > \tfrac{1}{2}d(G) - 1$, otherwise deleting a minimum degree vertex or contracting an edge with at most $\tau(G)$ common neighbours does not lower the average degree. This lower bound on $\tau(G)$ is the extra information gained when given a lower bound on the average degree instead of a lower bound on the minimum degree, and it is crucial to proving our result.

We prove the following theorem.
\begin{theorem}
\label{theorem:minimal}
For integers $k \geq 6$ and $r \geq 3$, every minimal graph with average degree at least $\tfrac{4}{3}kr$ contains a set of $k$ disjoint cycles, each with at least $r$ vertices.
\end{theorem}

Theorem~\ref{theorem:cycles} follows from Theorem~\ref{theorem:minimal} since every graph $G$ contains a minor $G'$ such that $G'$ is minimal and $d(G') \geq d(G)$. The proof of Theorem~\ref{theorem:minimal} is presented in the following section. It depends on a few technical lemmas, which we present in Section~\ref{sec:anc}.

\section{Proof of Theorem~\ref{theorem:minimal}}
\label{section:main}

Assume for the sake of a contradiction that there is a minimal graph $G$ with $d(G) \geq \tfrac{4}{3}kr$ that does not contain $k$ disjoint cycles of order at least $r$. For the sake of simplicity, we allow $K_2$ to be thought of as a $2$-vertex cycle, and $K_1$ as a $1$-vertex cycle. Let $\mathcal{C}$ be a collection of disjoint cycles in $G$ each with order at most $r$. Let $\mathcal{C}(i)$ denote the cycles of order $i$ in $\mathcal{C}$ for $1 \leq i \leq r$.

Choose $\mathcal{C}$ such that $|\mathcal{C}(r)|$ is maximised, then $|\mathcal{C}(r-1)|$ is maximised, and so on until $|\mathcal{C}(1)|$ is maximised.
Let $\mathcal{U}$ denote $\mathcal{C}(1) \cup \mathcal{C}(2) \cup \dots \cup \mathcal{C}(r-1)$.
If $|\mathcal{C}(r)| \geq k$, then $G$ contains at least $k$ disjoint cycles of order $r$, giving our desired contradiction. Hence we may assume that $|\mathcal{C}(r)| \leq k-1$.

Let $V(\mathcal{U}) := \bigcup_{C \in \mathcal{U}} V(C)$ and $V(\mathcal{C}(r)) := \bigcup_{C \in \mathcal{\mathcal{C}}(r)} V(C)$. Also, let $|C|:=|V(C)|$ for a cycle $C$, and let $|P|:=|V(P)|$ for a path $P$. 

\begin{claim}
\label{claim:truepartition}
Every vertex of $G$ is in a cycle of $\mathcal{C}$.
\end{claim}
\begin{proof}
Assume there is some $v \in V(G)$ that is not in a cycle of $\mathcal{C}$. Add $v$ to $\mathcal{C}$ as a cycle of length 1. This clearly gives a better choice of $\mathcal{C}$, contradicting our initial choice.
\end{proof}

\begin{claim}
\label{claim:upstairstriangle}
If $C$ is a cycle of $\mathcal{U}$ and $vw$ is an edge of $C$, then $v$ and $w$ do not have a common neighbour in any other cycle of $\mathcal{U}$ with order at most $|C|$.
\end{claim}
\begin{proof}
Assume otherwise, and let $v,w$ have a common neighbour $x$ in a cycle $D \in \mathcal{U}$, where $|D| \leq |C|$. Replace the edge $vw$ in $C$ with the path $v,x,w$, and remove $D$ from $\mathcal{C}$. Now $\mathcal{C}$ has lost a cycle of order $|C|$ and a cycle of order $|D|$, but has gained a cycle of order $|C|+1$. This gives a better choice of $\mathcal{C}$ since $|C| \leq r-1$.  
\end{proof}

\begin{claim}
\label{claim:otherupstairstriangle}
If $C'$ is a cycle of $\mathcal{U}$, $vw$ is an edge of $C'$, and $C$ is a cycle of $\mathcal{U}-\{C'\}$ with $|C| \geq |C'|$, then $v$ and $w$ have at most $\tfrac{1}{3}|C|$ common neighbours in $C$.
\end{claim}
\begin{proof}
Assume otherwise. Then there exist two vertices $u_1,u_2$ of $C$ such that $u_2$ is either one or two vertices clockwise from $u_1$. In the first case, $v$ is a common neighbour of $u_1$ and $u_2$, which contradicts Claim~\ref{claim:upstairstriangle}. In the second case, let $u_3$ be the vertex between $u_1$ and $u_2$. We construct a new cycle $D$ from $C$ by removing $u_3$ and adding the path $u_1,w,v,u_2$ between $u_1,u_2$. The cycle $D$ has order $|C|-1+2=|C|+1$. Replacing $C$ and $C'$ in $\mathcal{C}$ with $D$ gives a better choice of $\mathcal{C}$ since $|D| \leq r$.
\end{proof}

\begin{claim}
\label{claim:nolowedge}
If $x,y$ are 1-vertex cycles of $\mathcal{U}$, then there is no edge $xy$.
\end{claim}
\begin{proof}
If the edge $xy$ exists, then replace $x$ and $y$ in $\mathcal{C}$ with $xy$, giving a better choice of $\mathcal{C}$.
\end{proof}

Given an edge $vw$ of a cycle in $\mathcal{U}$, let $W(vw) := N(v) \cap N(w) \cap V(\mathcal{C}(r))$. Given a vertex $v$ of a cycle in $\mathcal{U}$, let $W(v) := N(v) \cap V(\mathcal{C}(r))$. 

We now present some key claims that use lemmas from Section~\ref{sec:anc}.
\begin{claim}
\label{claim:firstroute}
Let $C_1,C_2 \in \mathcal{U}$ be cycles such that $|C_1| \geq |C_2| \geq 2$, and let $vw,xy$ be edges of $C_1,C_2$ respectively. At least one of $W(vw),W(xy)$ contains at most $\tfrac{2}{3}|V(\mathcal{C}(r))|$ vertices.
\end{claim}
\begin{proof}
Assume for the sake of a contradiction that $|W(vw)| > \tfrac{2}{3}|V(\mathcal{C}(r))|$ and $|W(xy)| > \tfrac{2}{3}|V(\mathcal{C}(r))|$. Apply Lemma~\ref{lemma:reroute} with $\mathcal{F} := \mathcal{C}(r)$, $q_i := r-|C_2|$ for all $i=1,\dots,t$, $S:=W(xy)$ and $T:=W(vw)$. Thus there exists a path $P$ in some $F_i$ with $r-|C_2|$ vertices and both end vertices in $W(xy)$, together with a vertex $u$ of $W(vw)$ in $V(F_i)-V(P)$.

Construct cycle $C_1'$ from $C_1$ by removing the edge $vw$ and adding vertex $u$ and edges $vu,uw$. Construct cycle $C_2'$ from $C_2$ by removing the edge $xy$ and adding the path $P$ together with an edge from $x$ to one end vertex of $P$ and an edge from $y$ to the other end vertex. Thus $|C_1'|=|C_1|+1$ and $|C_2'|=|C_2|+r-|C_2|=r$. Hence if we remove $C_1,C_2,F_i$ from $\mathcal{C}$ and add $C_1',C_2'$ we lose cycles of order $|C_1|,|C_2|$ and $r$ but gain cycles of order $r$ and $|C_1|+1$. Since $r-1 \geq |C_1| \geq |C_2|$, this contradicts our initial choice of $\mathcal{C}$.
\end{proof}

This technique also works when one or both of the cycles is just a single vertex. The proofs of the next two claims are almost identical to the proof of Claim~\ref{claim:firstroute}, so we omit them.
\begin{claim}
\label{claim:firstroutea}
Let $C_1,C_2 \in \mathcal{U}$ be cycles such that $|C_1| \geq 2$ and $|C_2| = 1$, let $vw$ be an edge of $C_1$, and let $V(C_2) = \{x\}$. At least one of $W(vw),W(x)$ contains at most $\tfrac{2}{3}|V(\mathcal{C}(r))|$ vertices.
\end{claim}

\begin{claim}
\label{claim:firstrouteb}
Let $C_1,C_2 \in \mathcal{U}$ be cycles such that $V(C_1)=\{v\}$ and $V(C_2) =\{x\}$. At least one of $W(v),W(x)$ contains at most $\tfrac{2}{3}|V(\mathcal{C}(r))|$ vertices.
\end{claim}

Say a cycle of $\mathcal{U}$ is \emph{big} if it contains greater than $\tfrac{2}{3}r$ vertices, otherwise it is \emph{small}.

\begin{claim}
\label{claim:twocycles}
The set $\mathcal{U}$ contains at least two big cycles or two small cycles. In particular $|\mathcal{U}| \geq 2$.
\end{claim}
\begin{proof}
Assume for the sake of a contradiction that $\mathcal{U}$ contains at most one big cycle and one small cycle. By Claim~\ref{claim:truepartition}, $V(G) = V(\mathcal{U}) \cup V(\mathcal{C}(r))$. The cycles of $\mathcal{C}(r)$ contain at most $(k-1)r$ vertices in total. Every big cycle of $U$ contains at most $r-1$ vertices, and every small cycle contains at most $\tfrac{2}{3}r$ vertices. Thus $|V(G)| \leq (k-1)r + r -1 + \tfrac{2}{3}r=(k+\tfrac{2}{3})r -1$. Hence $\tfrac{4}{3}kr \leq d(G) \leq (k+\tfrac{2}{3})r-2$, implying $\tfrac{1}{3}kr \leq \tfrac{2}{3}r - 2$. This is a contradiction as $k \geq 2$.
\end{proof}

%
\begin{claim}
\label{claim:kminus1cycles}
$|\mathcal{C}(r)| = k-1$.
\end{claim}  
\begin{proof}
Assume $|\mathcal{C}(r)| \leq k-2$, for the sake of a contradiction. Thus $|V(\mathcal{C}(r))| \leq (k-2)r$. Let $C_1,C_2$ be the two largest cycles of $\mathcal{U}$ such that $|C_1| \geq |C_2|$. If $|C_1| \geq 2$, let $vw$ be an edge of $C_1$ and let $A:=W(vw)$. If $|C_1|=1$, then let $v$ be the only vertex of $C_1$ and let $A:=W(v)$. Similarly, let $xy$ be an edge of $C_2$ and set $B:=W(xy)$, unless $C_2$ contains only one vertex $x$ and $B:=W(x)$. We now show $|A|,|B| > \tfrac{2}{3}|V(\mathcal{C}(r))|$, which contradicts one of Claims~\ref{claim:firstroute}, \ref{claim:firstroutea} or \ref{claim:firstrouteb}.

If $|C_1|=1$, then every cycle in $\mathcal{U}$ is a $1$-cycle, and by Claim~\ref{claim:nolowedge}, $V(\mathcal{U})$ is an independent set, implying $A=N(v)$ and $|A| \geq \delta(G) > \tfrac{1}{2}d(G) > \tfrac{2}{3}|V(\mathcal{C}(r))|$, as desired. Otherwise, since $C_1$ is the largest cycle in $\mathcal{U}$, Claim~\ref{claim:upstairstriangle} implies $A$ contains all common neighbours of $v$ and $w$ except those in $C_1$ itself. Hence $|A| \geq \tau(G) - |C_1| + 2 > \tfrac{1}{2}d(G) - 1 - (r-1) + 2 = \tfrac{1}{2}d(G) -r + 2 \geq \tfrac{2}{3}kr - r +2 = \tfrac{2}{3}(k-2)r + \tfrac{4}{3}r -r+2 > \tfrac{2}{3}|V(\mathcal{C}(r))|$, as desired.

If $|C_2|=1$, then by Claim~\ref{claim:nolowedge}, $B$ contains all neighbours of $x$ except those in $C_1$. If $x$ is adjacent to more than $\tfrac{1}{2}|C_1|$ vertices of $C_1$, then two of those neighbours are themselves adjacent, contradicting Claim~\ref{claim:upstairstriangle}. Hence $|B| \geq \delta(G) - \tfrac{1}{2}|C_1| > \tfrac{2}{3}kr - \tfrac{1}{2}r = \tfrac{2}{3}(k-2)r + \tfrac{4}{3}r - \tfrac{1}{2}r > \tfrac{2}{3}|V(\mathcal{C}(r))|$, as desired. Alternatively, if $|C_2| \geq 2$, then $B$ contains all common neighbours of $x$ and $y$ except those in $C_1$ and $C_2$ (by Claim~\ref{claim:upstairstriangle}). The cycle $C_1$ contains at most $\tfrac{1}{3}|C_1|$ common neighbours of $x$ and $y$ by Claim~\ref{claim:otherupstairstriangle}, and $C_2$ contains at most $|C_2|-2$ common neighbours. Thus $|B| \geq \tau(G) - \tfrac{1}{3}|C_1| - |C_2| + 2 > \tfrac{2}{3}kr - \tfrac{4}{3}r + 2 \geq \tfrac{2}{3}(k-2)r+2 > \tfrac{2}{3}|V(\mathcal{C}(r))|$, as desired. This gives our desired contradiction.
%
%
\end{proof}

Claims~\ref{claim:firstroute}, \ref{claim:firstroutea} or \ref{claim:firstrouteb} show how we apply Lemma~\ref{lemma:reroute}; we now show how we use Lemma~\ref{lemma:otherreroute}. Note that since $\tfrac{2}{3}r \geq 2$, each big cycle contains an edge.

\begin{claim}
\label{claim:secondroute}
Let $C_1,C_2 \in \mathcal{U}$ be big cycles such that $|C_1| \geq |C_2|$, and let $vw,xy$ be edges of $C_1,C_2$ respectively. If $r \geq 4$, then at least one of $W(vw)$ and $W(xy)$ contains at most $\tfrac{2}{3}(k-1)r - \tfrac{1}{3}r$ vertices.
\end{claim}
\begin{proof}
Assume for the sake of a contradiction that $r \geq 4$, $|W(vw)| > \tfrac{2}{3}(k-1)r - \tfrac{1}{3}r$ and $|W(xy)| > \tfrac{2}{3}(k-1)r - \tfrac{1}{3}r$. Apply Lemma~\ref{lemma:otherreroute} with $\mathcal{F} := \mathcal{C}(r)$, $q_i := r-|C_2|$ for $i=1,\dots,k-1$, $S:=W(xy)$ and $T:=W(vw)$. We must ensure that $1 \leq q_i < \tfrac{1}{3}r$ for $i=1,\dots,k-1$ to be able to apply this.
Since $|C_2| \leq r-1$, it follows $q_i \geq 1$. Since $C_2$ is big, $|C_2| > \tfrac{2}{3}r$, and thus $q_i < \tfrac{1}{3}r$, as required.

The first outcome of Lemma~\ref{lemma:otherreroute} is identical to the result of Lemma~\ref{lemma:reroute}, and as such we get a contradiction by the same argument as in Claim~\ref{claim:firstroute}. Thus we only consider the second outcome. Let $F_i$ be the cycle of $\mathcal{C}(r)$ containing $P$ and $Q$. Construct the cycle $C_1'$ by removing the edge $vw$ from $C_1$ and adding the path $Q$, together with an edge from $v$ to one end vertex of $Q$ and an edge from $w$ to the other end vertex. Similarly, construct $C_2'$ from $C_2$ by removing the edge $xy$ and adding the path $P$, an edge from $x$ to one end vertex of $P$ and an edge from $y$ to the other end vertex of $P$. Thus $|C_1'|,|C_2'| > \tfrac{2}{3}r + \tfrac{1}{3}r = r$. Hence $(\mathcal{C}(r) - \{F_i\}) \cup \{C_1',C_2'\}$ is a set of $k$ cycles, where $k-2$ have order exactly $r$ and two have order $> r$. This contradicts our initial assumption about $G$.
\end{proof}

The following claim will also be helpful.

\begin{claim}
\label{claim:restrictneighbours}
Let $C_1,C_2 \in \mathcal{U}$ be cycles such that $|C_1| \geq |C_2| \geq 2$ and let $vw$ be an edge of $C_2$. If $v,w$ have $q \geq 2$ common neighbours in $V(C_1)$ then $|C_2| \leq \tfrac{|C_1|}{q}-1$.
\end{claim}
\begin{proof}
Label the common neighbours of $v,w$ in $V(C_1)$ clockwise by $x_1,\dots,x_q$. Let $P_i$ be the clockwise path from $x_i$ to $x_{i+1}$ inclusive (where $i+1$ is taken modulo $q$.) Hence $\sum_{i=1}^{q} |P_i| = |C_1|+q$, and thus we may fix $i$ so that $P_i$ contains at most $\tfrac{|C_1|+q}{q}$ vertices. Let $Q$ denote the path between $v$ and $w$ in $C_2$ that contains all vertices of $C_2$.

Construct a cycle $C_1'$ from $C_1$ by removing the interior vertices of $P_i$ and adding the path $Q$ and the edges $vx_i$ and $wx_{i+1}$. If $|C_1'| > r$, then $C_1'$ together with the $k-1$ cycles of $\mathcal{C}(r)$ contradict our initial assumption about $G$. Now assume $|C_1'| \leq r$. If $|C_1'| > |C_1|$, then replacing $C_1,C_2$ in $\mathcal{C}$ with $C_1'$ gives a better choice of $\mathcal{C}$. Otherwise $|C_1'| \leq |C_1|$. Hence $|C_1| - (|P_i|-2) + |C_2| \leq |C_1|$, implying $|C_2| \leq |P_i| - 2 \leq \tfrac{|C_1|}{q}-1.$
\end{proof}

\begin{claim}
\label{claim:notwobig}
There is at most one big cycle in $\mathcal{U}$. 
\end{claim}
\begin{proof}
First suppose that $r \geq 4$. Let $C_1$ and $C_2$ be the two largest cycles in $\mathcal{U}$ such that $|C_1| \geq |C_2|$. Assume for the sake of a contradiction that $C_1$ and $C_2$ are both big. Recall both must contain an edge, and so let $A:=W(vw)$ and $B:=W(xy)$, where $vw,xy$ are edges of $C_1,C_2$ respectively. We now show that $|A|,|B| > \tfrac{2}{3}(k-1)r - \tfrac{1}{3}r$ which contradicts Claim~\ref{claim:secondroute}.

By Claim~\ref{claim:upstairstriangle}, $A$ contains every common neighbour of $v$ and $w$ except those in $C_1$ itself, of which there are at most $r-3$. Hence $|A| \geq \tau(G) - r +3 > \tfrac{2}{3}kr -1 - r+3 = \tfrac{2}{3}(k-1)r - \tfrac{1}{3}r+2$.

By Claim~\ref{claim:upstairstriangle}, $B$ contains every common neighbour of $x$ and $y$ except those in $C_1$ and in $C_2$. If there are at least two common neighbours of $x$ and $y$ in $C_1$, then by Claim~\ref{claim:restrictneighbours}, $|C_2| \leq \tfrac{|C_1|}{2}-1 < \tfrac{r}{2}-1$. However $|C_2| > \tfrac{2}{3}r$, so this gives a contradiction. Hence $x$ and $y$ have at most one common neighbour in $C_1$, and at most $r-3$ common neighbours in $C_2$. Thus $|B| \geq \tau(G) - 1 - r+3 > \tfrac{2}{3}(k-1)r - \tfrac{1}{3}r+1$, which is sufficient. 

Finally, when $r=3$, every cycle in $\mathcal{U}$ has at most $r-1=2$ vertices, but every big cycle in $\mathcal{U}$ contains greater than $\tfrac{2}{3}r = 2$ vertices, and thus there are no big cycles in $\mathcal{U}$.
\end{proof}

Claims~\ref{claim:twocycles} and \ref{claim:notwobig} together imply that $\mathcal{U}$ contains at most one big cycle and at least two small cycles.

\begin{claim}
\label{claim:Aisbig}
Let $C_1$ be the largest small cycle. Suppose $|C_1| \geq 2$, and let $vw$ be an edge of $C_1$. Then $|W(vw)| > \tfrac{2}{3}(k-1)r$.
\end{claim}
\begin{proof}
If there is a big cycle then, by Claim~\ref{claim:notwobig}, there is only one; label such a cycle $C_0$. By Claim~\ref{claim:upstairstriangle}, $W(vw)$ contains all common neighbours of $v$ and $w$ except those in $C_0$ and $C_1$. 

If $|C_1| \leq \tfrac{1}{3}r+1$, then there are at most $\tfrac{1}{3}|C_0|$ common neighbours of $v$ and $w$ in $C_0$ (by Claim~\ref{claim:otherupstairstriangle}) and at most $\tfrac{1}{3}r-1$ common neighbours in $C_1$. Thus $|W(vw)| > \tau(G) - \tfrac{1}{3}|C_0| - (\tfrac{1}{3}r-1) \geq \tfrac{2}{3}kr -1 - \tfrac{1}{3}|C_0| - \tfrac{1}{3}r+1 > \tfrac{2}{3}(k-1)r + \tfrac{2}{3}r - \tfrac{2}{3}r$, which is sufficient.

Otherwise $\tfrac{1}{3}r+1 < |C_1|$. Let $q$ be the number of common neighbours of $v$ and $w$ in $C_0$. 
If $q \geq 3$, then by Claim~\ref{claim:restrictneighbours}, $|C_1| \leq \tfrac{|C_0|}{q} - 1 = \tfrac{1}{3}r-1$, which is a contradiction. Hence $q \leq 2$. 
If $q=2$, then (by Claim~\ref{claim:restrictneighbours} again) $|C_1| \leq \tfrac{1}{2}r-1$, and hence $|W(vw)| \geq \tau(G) - q - (|C_1|-2) > \tfrac{2}{3}kr -1 - q+2 -|C_1| \geq \tfrac{2}{3}(k-1)r + \tfrac{2}{3}r - 1 - \tfrac{1}{2}r+1 > \tfrac{2}{3}(k-1)r$, as desired. 
Thus $q \leq 1$, and $|W(vw)| \geq \tau(G) - q - (|C_1|-2) > \tfrac{2}{3}kr -1 - q+2 -|C_1| \geq \tfrac{2}{3}(k-1)r + \tfrac{2}{3}r - |C_1| \geq \tfrac{2}{3}(k-1)r$, since $C_1$ is small and $|C_1| \leq \tfrac{2}{3}r$. Hence our lower bound on $|W(vw)|$ holds.
%
\end{proof}

Our goal now is to show that $\mathcal{U}$ contains few vertices; this shall give the final contradiction.

\begin{claim}
\label{claim:smallissmall}
If $C_1,C_2 \in \mathcal{U}$ are the two largest small cycles such that $|C_1| \geq |C_2|$, then $|C_2|=1$. 
\end{claim}
\begin{proof}
Assume for the sake of a contradiction that $|C_1| \geq |C_2| \geq 2$. Pick edges $vw$ in $C_1$ and $xy$ in $C_2$, and let $A:=W(vw)$ and $B:=W(xy)$. As we did in Claim~\ref{claim:kminus1cycles}, we show that $|A|,|B| > \tfrac{2}{3}(k-1)r$, which contradicts one of Claims~\ref{claim:firstroute}, \ref{claim:firstroutea} or \ref{claim:firstrouteb}. By Claim~\ref{claim:notwobig}, it is possible that one big cycle exists in $\mathcal{U}$; we denote this cycle by $C_0$.

By Claim~\ref{claim:Aisbig}, $|A| > \tfrac{2}{3}(k-1)r$. Now consider the lower bound on $|B|$. Let $p := \tfrac{2}{3}r-|C_2|$ and note $p \geq 0$ (since $C_2$ is small). The set $B$ contains all common neighbours of $x$ and $y$ except those in $C_0$ (the number of which we denote by $q_0$), those in $C_1$ (the number of which we denote by $q_1$) and at most $|C_2|-2$ in $C_2$. Thus $|B| \geq \tau(G) - q_0 - q_1 - |C_2|+2 > \tfrac{2}{3}(k-1)r + \tfrac{2}{3}r - q_0 - q_1 - |C_2|+1$. If $\tfrac{2}{3}r - q_0 - q_1 - |C_2|+1 \geq 0$, then we are done. So assume $p < q_0 + q_1 -1$. Since $p \geq 0$ we have $q_0 + q_1 \geq 2$. 

If $q_0 \geq 2$, then by Claim~\ref{claim:restrictneighbours}, $|C_2| \leq \tfrac{|C_0|}{q_0}-1$. Since $|C_2| \geq 2$ and $|C_0| \leq r-1$, it follows that $3q_0 + 1 \leq r$. Also $\tfrac{2}{3}r-p = |C_2| \leq \tfrac{|C_0|}{q_0} - 1$. Hence $\tfrac{2}{3}r-p+1 \leq \tfrac{r-1}{q_0}$ and thus $2q_0^2 + \tfrac{2}{3}q_0 -3q_0 -pq_0 + q_0 = (\tfrac{2}{3}q_0-1)(3q_0 + 1) - pq_0 + q_0 + 1 \leq (\tfrac{2}{3}q_0-1)r -pq_0 + q_0 +1 \leq 0$. Thus $2q_0 + \tfrac{2}{3} - 3 - p + 1 \leq 0$ and $q_0 \leq \tfrac{1}{2}p + \tfrac{2}{3}.$

Similarly, if $q_1 \geq 2$, by Claim~\ref{claim:restrictneighbours}, $|C_2| \leq \tfrac{|C_1|}{q_1} -1$. Since $|C_2| \geq 2$ and $|C_1| \leq \tfrac{2}{3}r$, it follows $\tfrac{9}{2}q_1 \leq r$. Also $\tfrac{2}{3}r - p = |C_2| \leq \tfrac{|C_1|}{q_1} - 1$. Thus $\tfrac{2}{3}r-p+1 \leq \tfrac{2r}{3q_1}$, and so $\tfrac{2}{3}(q_1-1)(\tfrac{9}{2}q_1) - pq_1 + q_1 \leq \tfrac{2}{3}(q_1-1)r -pq_1 + q_1 \leq 0$. Hence $3(q_1 - 1) - p + 1 \leq 0$ and $q_1 \leq \tfrac{1}{3}p + \tfrac{2}{3}$.

We now show that $p < 2$. If $q_0 \geq 2$ and $q_1 \geq 2$, then $p < q_0 + q_1 -1 \leq \tfrac{1}{2}p + \tfrac{2}{3} + \tfrac{1}{3}p + \tfrac{2}{3} - 1 = \tfrac{5}{6}p + \tfrac{1}{3}$, and thus $p < 2$. If $q_0 \geq 2$ and $q_1 \leq 1$, then $p < q_0 \leq \tfrac{1}{2}p + \tfrac{2}{3}$, and $p < \tfrac{4}{3} < 2$. If $q_1 \geq 2$ and $q_0 \leq 1$, then $p < q_1 \leq \tfrac{1}{3}p + \tfrac{2}{3}$ and $p < 1$. Finally, if $q_0 \leq 1$ and $q_1 \leq 1$ then $p < 1$. Hence $p < 2$.

If $q_0 \geq 2$ or $q_1 \geq 2$, then $q_0 < \tfrac{5}{3} < 2$ or $q_1 < \tfrac{4}{3} < 2$ respectively, and thus $q_0 \leq 1$ and $q_1 \leq 1$. Since $q_0 + q_1 \geq 2$, it follows $q_0=q_1=1$, and $p < 1$. Hence $|C_2| > \tfrac{2}{3}r-1$. Since $|C_2| \leq \tfrac{2}{3}r$ and $|C_2|$ is an integer, $|C_2| \leq \floor{\tfrac{2}{3}r}$, and together with our lower bound, $|C_2| = \floor{\tfrac{2}{3}r}$. Since $|C_1| \geq |C_2|$ and $C_1$ is also small, $|C_1|=|C_2|$. However, since $q_1 = 1$, this contradicts Claim~\ref{claim:upstairstriangle}. 
\end{proof}

We now know that $\mathcal{U}$ contains
\begin{itemize*}
\item at most one big cycle (by Claim~\ref{claim:notwobig}),
\item at least two small cycles (by Claims~\ref{claim:notwobig} and \ref{claim:twocycles}), and
\item every small cycle, except possibly the largest, is a $1$-cycle (by Claim~\ref{claim:smallissmall}).
\end{itemize*}
Let $\mathcal{I}$ be the set of all small cycles of $\mathcal{U}$ other than the largest. Since these are all $1$-cycles, we interpret $\mathcal{I}$ as a set of vertices; $\mathcal{I}$ is an independent set by Claim~\ref{claim:nolowedge}. Let $C_1$ be the largest small cycle in $\mathcal{U}$. 

\begin{claim}
\label{claim:structure}
$ $
\begin{enumerate*}
\item[(a)] there exists a big cycle in $\mathcal{U}$, denoted $C_0$,
\item[(b)] every $u \in \mathcal{I}$ has more than $\tfrac{1}{3}|C_0|$ neighbours in $C_0$,
\item[(c)] every $u \in \mathcal{I}$ has at least two neighbours in $C_1$, and 
\item[(d)] $|C_1| > \tfrac{1}{3}r$.
\end{enumerate*}
\end{claim}
\begin{proof}
Assume, for the sake of a contradiction, that at least one of the following holds: 
\begin{enumerate*}
\item[(a)] $\mathcal{U}$ contains no big cycle,
\item[(b)] there exists a big cycle $C_0$ and some $u \in \mathcal{I}$ has at most $\tfrac{1}{3}|C_0|$ neighbours in $C_0$,
\item[(c)] there exists a big cycle $C_0$ and some $u \in \mathcal{I}$ has at most one neighbour in $C_1$, or
\item[(d)] $|C_1| \leq \tfrac{1}{3}r.$
\end{enumerate*}
If $C_1$ contains an edge $vw$, let $A:=W(vw)$, otherwise let $A:=W(v)$ where $V(C_1) = \{v\}$. Let $B:=W(u)$. Once again we show that $|A|,|B| > \tfrac{2}{3}(k-1)r$, which contradicts either Claim~\ref{claim:firstroutea} or \ref{claim:firstrouteb}.

If $|C_1| \geq 2$, then by Claim~\ref{claim:Aisbig}, $|A|$ is sufficiently large. If $|C_1|=1$, then $A$ contains all neighbours of $v$ except those in $C_0$ (by Claim~\ref{claim:nolowedge}). By Claim~\ref{claim:upstairstriangle}, there are at most $\tfrac{1}{2}|C_0| \leq \tfrac{1}{2}r$ of these. Hence $|A| \geq \delta(G) - \tfrac{1}{2}r > \tfrac{2}{3}kr - \tfrac{1}{2}r = \tfrac{2}{3}(k-1)r + \tfrac{1}{6}r$, as desired.

Now consider $|B|$. The set $B$ contains all vertices adjacent to $u$ except those in $C_0$ and those in $C_1$ (since $\mathcal{I}$ is an independent set); say there are $q_0$ neighbours of $u$ in $C_0$ and $q_1$ neighbours of $u$ in $C_1$. Thus $|B| \geq \delta(G) - q_0 - q_1 > \tfrac{2}{3}kr - q_0 -q_1 = \tfrac{2}{3}(k-1)r + \tfrac{2}{3}r-q_0-q_1$. If $\tfrac{2}{3}r-q_0-q_1 \geq 0$ we are done. So assume $q_0 + q_1 > \tfrac{2}{3}r.$

By Claim~\ref{claim:upstairstriangle}, it follows that $q_0 \leq \tfrac{1}{2}|C_0|$ and $q_1 \leq \tfrac{1}{2}|C_1|$. In (a) or (b), $\tfrac{2}{3}r < q_0 + q_1 \leq \tfrac{1}{3}(r-1) + \tfrac{1}{2}(\tfrac{2}{3}r) = \tfrac{2}{3}r - \tfrac{1}{3}$, since $C_1$ is small; this is a contradiction. In (c), $\tfrac{2}{3}r < q_0 + q_1 \leq \tfrac{1}{2}(r-1) +1$, which implies $r<3$, again a contradiction. In (d), $\tfrac{2}{3}r < q_0 + q_1 \leq \tfrac{1}{2}(r-1) + \tfrac{1}{6}r = \tfrac{2}{3}r - \tfrac{1}{2}$, yet again a contradiction.  Thus we have proven our desired lower bounds on $|A|$ and $|B|$ in all cases.
\end{proof}

By Claim~\ref{claim:twocycles} and Claim~\ref{claim:notwobig}, $\mathcal{I} \neq \emptyset$. Our final step is to bound $|\mathcal{I}|$.

\begin{claim}
$|\mathcal{I}| = 1$.
\end{claim}
\begin{proof}
Assume $|\mathcal{I}| \geq 2$, for the sake of a contradiction. Let $x,x' \in \mathcal{I}$. Let $S,T$ be the neighbours of $x,x'$ in $C_0$ respectively. By Claim~\ref{claim:structure}, $|S|,|T| > \tfrac{1}{3}|C_0|$. Given that $|C_0| > \tfrac{2}{3}r \geq 2$, it follows from Lemma~\ref{lemma:lemmafordouble} that there exists a path $P$ in $C_0$ with one end vertex in $S$, the other in $T$, and $2 \leq |P| \leq \tfrac{1}{6}|C_0| + 4 < \tfrac{1}{6}r+4$. Let $v$ be the end vertex of $P$ in $S$ and $w$ the end vertex of $P$ in $T$. Since $x,x'$ both have two or more neighbours in $C_1$ (again by Claim~\ref{claim:structure}), there is a neighbour $a$ of $x$ and a neighbour $b \neq a$ of $x'$. Let $Q$ be the longer path in $C_1$ between $a$ and $b$; since, by Claim~\ref{claim:structure}, $|C_1| > \tfrac{1}{3}r$, it follows $|Q| > \tfrac{1}{6}r$.

Create the cycle $C$ in the following way. Start with $C_0$ and remove the interior vertices of $V(P)$. Add edges $vx$ and $wx'$, and then the edges $xa$ and $x'b$. Finally add the path $Q$ between $a$ and $b$. Thus $|C| = |C_0| - (|P|-2) + 2 + |Q| = |C_0| - |P| + 4 + |Q| > |C_0| - (\tfrac{1}{6}r+4)+4 + \tfrac{1}{6}r = |C_0|$. If $|C| \leq r$, then removing $C_0,C_1,\{x\}$ and $\{x'\}$ from $\mathcal{C}$ and adding $C$ gives a better choice of $\mathcal{C}$. Otherwise, $\mathcal{C}(r) \cup \{C\}$ is a set of $k$ cycles, where $k-1$ have order exactly $r$ and one has order $>r$; this contradicts our initial assumption about $G$. 
\end{proof}

Now we know that $\mathcal{U}$ contains three cycles: big cycle $C_0$, small cycle $C_1$ and a single $1$-cycle in $\mathcal{I}$. Hence $|V(\mathcal{U})| = |C_0| + |C_1| + 1 \leq (r-1) + \tfrac{2}{3}r + 1 = \tfrac{4}{3}r$. Given that $|V(\mathcal{C}(r))| = (k-1)r$, it follows by Claim~\ref{claim:truepartition} that $|V(G)| \leq (k-1)r + \tfrac{4}{3}r = kr + \tfrac{1}{3}r$, and $d(G) \leq kr + \tfrac{1}{3}r -1$. Thus $\tfrac{4}{3}kr \leq d(G) \leq kr + \tfrac{1}{3}r-1$, implying $0 \leq \tfrac{1}{3}(k-1)r \leq -1.$
This is the final contradiction, and Theorems~\ref{theorem:intermsofsubgraphs}, \ref{theorem:cycles} and \ref{theorem:minimal} are proven. Note that we have actually proven the following strengthening of Theorem~\ref{theorem:minimal}: for integers $k \geq 6$ and $r \geq 3$, every minimal graph with average degree at least $\tfrac{4}{3}kr$ contains a set of $k$ disjoint cycles, $k-2$ with exactly $r$ vertices and two with at least $r$ vertices.

\section{Ancillary Lemmas}
\label{sec:anc}

This section presents some technical lemmas that were used in Section~\ref{section:main}. Note that none of these results depend on the choice of $G$. Recall if $C$ is a cycle then $|C| := |V(C)|$ and if $P$ is a path then $|P|:= |V(P)|$. Also if $\mathcal{F}$ is a collection of disjoint cycles, then let $V(\mathcal{F}) := \bigcup_{F \in \mathcal{F}} V(F)$. In this section, all cycles contain at least three vertices; that is, we no longer treat $K_2$ and $K_1$ as cycles. (The lemmas in this section are applied to either $\mathcal{C}(r)$, which only contains cycle of length $r \geq 3$, or to big cycles of $\mathcal{U}$, which contain strictly more than $\tfrac{2}{3}r \geq 2$ vertices, so this is acceptable.)

\begin{lemma}
\label{lemma:reroute}
Let $\mathcal{F} = \{F_1,\dots,F_t\}$ be a collection of disjoint cycles and let $q_1,\dots,q_t$ be integers such that $1 \leq q_i \leq |F_i|-1$. If $S,T \subseteq V(\mathcal{F})$ and $|S|,|T| > \tfrac{2}{3}|V(\mathcal{F})|$, then there exists a path $P$ in some $F_i$ with exactly $q_i$ vertices, such that both end vertices of $P$ are in $S$ and there exists a vertex of $T$ in $V(F_i) - V(P)$.
\end{lemma}
\begin{proof}
Assume otherwise for the sake of a contradiction. Fix an orientation on every cycle in $F_1,\dots,F_t$. For each vertex $v \in V(F_i)$, let $f(v)$ denote the vertex $q_i-1$ vertices clockwise from $v$ in $F_i$, and let $P_v$ be the $q_i$-vertex path from $v$ clockwise to $f(v)$. Let $f(S) := \{f(v) : v \in S\}$. Since $f: V(\mathcal{F}) \rightarrow V(\mathcal{F})$ is a $1-1$ correspondence, $|f(S)|=|S|$. Let $\mathcal{P} := \{P_v : v,f(v) \in S\}$. If $P_v \in \mathcal{P}$, then $f(v) \in S \cap f(S)$. Conversely, if $w \in S \cap f(S)$, then $w = f(v)$ for some $v \in S$, and so $P_v \in \mathcal{P}$. Hence $|\mathcal{P}| = |S \cap f(S)|$. Thus,
\begin{align}
|V(\mathcal{F})| \geq |S \cup f(S)| = |S| + |f(S)| - |S \cap f(S)| = 2|S| - |\mathcal{P}| \label{eq:eq_first}.
\end{align} 

Let $x(v)$ denote the vertex counter-clockwise from $v$. Since $|P_v|=q_i \leq |F_i|-1$, the vertex $x(v)$ is not in $P_v$. Let $\mathcal{X} := \{x(v) : P_v \in \mathcal{P}\}$. As $x$ is $1-1$, $|\mathcal{X}| = |\mathcal{P}|$. If $\mathcal{X} \cap T \neq \emptyset$, then there is a path in $\mathcal{P}$, which has both end vertices in $S$, together with a vertex of $T$ in the same cycle as $P$ but not intersecting $P$ itself. Hence $\mathcal{X} \cap T = \emptyset$. Thus,
\begin{align}
|V(\mathcal{F})| \geq |\mathcal{X} \cup T| = |\mathcal{X}| + |T| - |\mathcal{X} \cap T| = |\mathcal{P}| + |T| \label{eq:eq_second}.
\end{align}
\eqref{eq:eq_first} and \eqref{eq:eq_second} imply it follows $2|V(\mathcal{F})| \geq 2|S|+|T| >  \tfrac{4}{3}|V(\mathcal{F})| + \tfrac{2}{3}|V(\mathcal{F})| = 2|V(\mathcal{F})|$, which is a contradiction.
\end{proof}

\begin{lemma}
\label{lemma:imprweirdcase}
Let $C$ be a cycle such that $|C| \geq 6$ and let $p := \floor{\tfrac{|C|}{3}}$. If $S,T \subseteq V(C)$ such that $|S| \geq p+1$ and $|T| \geq 2p+2$, then there exist disjoint paths $P,Q$ such that the end vertices of $P$ are both in $S$, the end vertices of $Q$ are both in $T$, and $|V(P)|,|V(Q)| \geq p+1$.
\end{lemma}
\begin{proof}
Note $|C| \in \{3p,3p+1,3p+2\}$. First, we shall show there is a path $A$ in $C$ with both end vertices in $S$ such that $p+1 \leq |A| \leq |C|+1-2p$. Label the vertices of $C$ clockwise $1,\dots,|C|$ so that, if possible, $3p+1,3p+2 \notin S$. Let $X_i := \{i,p+i,2p+i\}$ for $i=1,\dots,p$. If $|X_i \cap S| \geq 2$ for some $i$, then the shorter path between two vertices of $X_i \cap S$ has at least $p+1$ vertices, and at most $p+1+(|C|-3p)=|C|+1-2p$, as desired. Otherwise $|X_i \cap S| \leq 1$ for all $i$. If $3p+1,3p+2 \notin S$, then $|S| \leq p$, which is a contradiction.

Hence we may assume that at least one of $3p+1$ and $3p+2$ is in $S$. If $|C|=3p$, no vertex is labelled $3p+1$ or $3p+2$, and this case cannot occur. If $|C|=3p+1$, then $3p+1 \in S$ and, by our choice of labelling, all vertices of $C$ are in $S$. But then it is trivial to construct a $(p+1)$-vertex path $A$ with both end vertices in $S$. If $|C|=3p+2$, then by our choice of labelling, there do not exist two adjacent vertices such that neither is in $S$. Let $v$ be a vertex of $S$, and let $w,w'$ be the two vertices $p,p+1$ vertices clockwise from $v$. At least one of $w,w'$ is in $S$, and together with $v$ this gives a path with both end vertices in $S$ of either $p+1$ or $p+2$ vertices, as required.

Thus there is a path $A$ with $p+1 \leq |A| \leq |C|+1-2p$ and both end vertices in $S$. If there are multiple choices for $A$ choose one of shortest length. If $|T-V(A)| \geq p+1$, set $P=A$ and let $Q$ be the longest path in $C-V(A)$ with both end vertices in $T$; such a path must contain all the vertices of $T-V(A)$, and thus $|Q| \geq p+1$, as desired. Hence we may assume that $|T-V(A)| \leq p$. Since $|T| \geq 2p+2$, we have $|V(A) \cap T| \geq p+2$ and $|A| \geq p+2$. 

If $|C|=3p$, then $|V(A) \cap T| \leq |A|=p+1$, which is a contradiction. Hence $|C| \geq 3p+1$. Let $v$ be the counter-clockwise end vertex of $A$ and $w$ the clockwise end vertex. Label the two vertices clockwise from $v$ by $v',v''$ and the two vertices counter-clockwise from $w$ by $w',w''$. These vertices are all in $A$ since $|A| \geq p+2 \geq 4$. If $|A|=p+2$, then $v',w' \notin S$, or else we could replace $A$ by a smaller path that still contains $p+1$ vertices. Thus $|V(A) \cap S| \leq p$ (since $|A| \geq 4$, $v'$ and $w'$ are distinct vertices). Alternatively, if $|A|=p+3$, then by the same argument $v',w',v'',w'' \notin S$, and $|V(A) \cap S| \leq p$ (since $|A|=p+3 \geq 5$ and thus $|\{v',w',v'',w''\}| \geq 3$). In either case, since $|S| \geq p+1$, at least one vertex of $S$ is not in $A$; label one such vertex $x$. If the clockwise path from $w$ to $x$ and the clockwise path from $x$ to $v$ both contain at most $p$ vertices, then the clockwise $v$ to $w$ path contains at most $2p-1$ vertices, and hence $|C|- (2p-1)+2 \leq |A| \leq |C|+1-2p$, which is a contradiction. Without loss of generality, say the clockwise path from $w$ to $x$ contains at least $p+1$ vertices, and let this be $P$. Since $|V(A) \cap T| \geq p+2$, it follows $|(V(A)-\{w\}) \cap T| \geq p+1$, so the longest path with both end vertices in $T$ contained within $V(A)-\{w\}$ contains at least $p+1$ vertices, and we set it as $Q$. This proves the lemma.
\end{proof}

\begin{lemma}
\label{lemma:otherreroute}
Let $\mathcal{F} = \{F_1,\dots,F_{k-1}\}$ be a collection of $k-1 \geq 5$ disjoint cycles of order $r \geq 4$, and let $q_1,\dots,q_{k-1}$ be integers such that $1 \leq q_i < \tfrac{1}{3}r$. If $S,T \subseteq V(\mathcal{F})$ and $|S|,|T| > \tfrac{2}{3}(k-1)r - \tfrac{1}{3}r$, then at least one of the following cases holds:
\begin{itemize*}
\item There exists a path $P$ in some $F_i$ with exactly $q_i$ vertices, such that both end vertices of $P$ are in $S$, and there exists a vertex of $T$ in $V(F_i) - V(P)$.
\item There exist disjoint paths $P,Q$ in some $F_i$ such that the end vertices of $P$ are in $S$, the end vertices of $Q$ are in $T$, and $|P|,|Q| > \tfrac{1}{3}r.$
\end{itemize*}
\end{lemma}
\begin{proof}
Note that the first part of this proof is very similar to Lemma~\ref{lemma:reroute}.
For the sake of a contradiction, assume neither outcome occurs. For each $v \in V(F_i)$, let $f(v)$ denote the vertex $q_i -1$ vertices clockwise from $v$, and let $P_v$ be the $q_i$-vertex path from $v$ clockwise to $f(v)$. Let $f(S):=\{f(v) : v \in S\}$, and note $|f(S)|=|S|$ since $f$ is a $1-1$ correspondence. Let $\mathcal{P} := \{P_v : v,f(v) \in S\}$. If $w \in S \cap f(S)$, then $w = f(v)$ such that $P_v \in \mathcal{P}$. Conversely, if $P_v \in \mathcal{P}$ then $f(v) \in S \cap f(S)$. Hence $|\mathcal{P}| = |S \cap f(S)|$. Thus,
\begin{align}
(k-1)r \geq |S \cup f(S)| = |S| + |f(S)| - |S \cap f(S)| = 2|S| - |\mathcal{P}|.\label{eq:eq_third}
\end{align}

Let $x(v)$ denote the vertex counter-clockwise from $v$. Then $x(v) \notin V(P_v)$ since $q_i < \tfrac{1}{3}r < r$. Let $\mathcal{X} := \{x(v) : P_v \in \mathcal{P}\}$. As $x$ is $1-1$, $|\mathcal{P}|=|\mathcal{X}|$. Let $X_i := \mathcal{X} \cap V(F_i)$ and let $\mathcal{P}_i := \{P_v \in \mathcal{P}: v \in V(F_i)\}$. We classify each of $F_1, \dots, F_{k-1}$ as one of the four following types. If $\mathcal{P}_i \neq \emptyset$ and there exists at least one vertex $w \in V(F_i)$ such that  $w \in V(P_v)$ for all $P_v \in \mathcal{P}_i$, then $F_i$ has \emph{type 1}. If $\mathcal{P}_i \neq \emptyset$ but there is no such $w$, then $F_i$ has \emph{type 2}. Otherwise $\mathcal{P}_i = \emptyset$; if $|S \cap V(F_i)| \leq \tfrac{1}{3}r$, then $F_i$ has \emph{type 3}, otherwise it has \emph{type 4}.

We now define a set $Y_i$ for each $F_i$. If $F_i$ has type 2, 3 or 4, let $Y_i := X_i$. (Note if $F_i$ has type 3 or 4, then $X_i = \emptyset$.) If $F_i$ has type 1, label the vertices of $F_i$ clockwise by $1,2,\dots,r$, starting one vertex clockwise from $w$, where $w \in V(P_v)$ for all $P_v \in \mathcal{P}_i$. Let $v$ be the vertex of minimum index such that $P_v \in \mathcal{P}_i$. Let $J := V(F_i) - P_v$. The set $J$ starts at $x(v)$ and travels counter-clockwise; it may reach as far as vertex $1$, but does not include vertex $r$ by our choice of labelling (since $w = r$ is in $P_v$). Also, $J \cap X_i = \{x(v)\}$, by our choice of $v$. As $q_i < \tfrac{1}{3}r$, we have $|J| > \tfrac{2}{3}r$. Let $Y_i = X_i \cup J$, and note $|Y_i| > |X_i| + \tfrac{2}{3}r - 1$ when $F_i$ is type 1. 

Define $\mathcal{Y} := \cup_{i=1}^{k-1} Y_i$. If $\mathcal{Y} \cap T \neq \emptyset$, then there exists a vertex of $T$ in a cycle of type 1 or 2 that avoids a path of $\mathcal{P}$ in the same cycle; together this satisfies the first outcome. Now assume that $\mathcal{Y} \cap T = \emptyset$. Let $\theta_1$ denote the number of type 1 cycles. First suppose that $\theta_1 \geq 2$. Thus $|\mathcal{Y}| > |\mathcal{X}| + 2(\tfrac{2}{3}r-1)$. Hence $(k-1)r \geq |\mathcal{Y} \cup T| = |\mathcal{Y}| + |T| > |\mathcal{X}| + 2(\tfrac{2}{3}r-1) + |T| = |\mathcal{P}| + |T| + \tfrac{4}{3}r-2.$

Let $r = 3a +b$ where $a,b \in \mathbb{Z}$ and $a \geq 1$ and $b \in \{0,1,2\}$. Thus $(k-1)r > |\mathcal{P}| + |T| + 4a + \tfrac{4}{3}b -2$, and so 
\begin{align}
(k-1)r \geq |\mathcal{P}| + |T| + 4a + b + \floor{\tfrac{1}{3}b}+1-2 = |\mathcal{P}| + |T| + 4a + b - 1. \label{eq:eq_fourth}
\end{align}

Inequalities~\eqref{eq:eq_third} and \eqref{eq:eq_fourth} imply
\begin{align*}
2(k-1)r \geq 2|S| - |\mathcal{P}| + |\mathcal{P}| + |T| + 4a +b -1 > 2(k-1)r - r + 4a+b-1.
\end{align*}
Thus $3a+b = r \geq 4a+b$, contradicting $a \geq 1$. Hence we may assume $\theta_1 \leq 1$.

Define $\theta_2,\theta_3,\theta_4$ to be the number of type $2,3,4$ cycles respectively. Note $\theta_1 + \theta_2 + \theta_3 + \theta_4 = k-1$. Let $S_i := S \cap V(F_i)$ and $T_i := T \cap V(F_i)$. We now prove some bounds on $|S_i|$ and $|T_i|$.
\begin{itemize}
\item If $F_i$ has type 1, then $|S_i| \leq r$ trivially, and since $T_i \cap Y_i = \emptyset$ and $Y_i = X_i \cup J$ and $|J| > \tfrac{2}{3}r$, it follows that $|T_i| < \tfrac{1}{3}r$.
\item If $F_i$ has type 2, then $|S_i| \leq r$ trivially. Since for every $w \in V(F_i)$ there exists some $P_v \in \mathcal{P}_i$ such that $w \notin V(P_v)$, if $w \in T_i$ then $w$ together with $P_v$ satisfy the first outcome. Thus we may assume that $T_i=\emptyset$.
\item If $F_i$ has type 3, then $|S_i| \leq \floor{\tfrac{1}{3}r}$ by definition, and $|T_i| \leq r$ trivially.
\item If $F_i$ has type 4, then $|S_i| > \tfrac{1}{3}r$ by definition. If $r \leq 6$, then $q_i < 2$, and hence $q_i=1$. But then $\mathcal{P}_i = S_i \neq \emptyset$, contradicting the definition of type 4. Hence $r > 6$. Let $p:=\floor{\tfrac{r}{3}}$, and note $|S_i| \geq p+1$. If $|T_i| \geq 2p+2$, then by Lemma~\ref{lemma:imprweirdcase} there exist disjoint paths $P,Q$ in $F_i$ such that the end vertices of $P$ are in $S_i \subseteq S$ and the end vertices of $Q$ are in $T_i \subseteq T$ and $|P|,|Q|\geq p+1 > \tfrac{r}{3}$. This would satisfy the second outcome, and hence we may assume that $|T_i| \leq 2p+1$.\newline If $r \geq 9$, then $|T_i| \leq \tfrac{2}{3}r+1 = \tfrac{7}{9}r - \tfrac{1}{9}r+1 \leq \tfrac{7}{9}r$. Alternatively $7 \leq r \leq 8$, and $|T_i| \leq 2p+1 = 2\floor{\tfrac{8}{9}} +1 = 5 < \tfrac{7}{9}(7) \leq \tfrac{7}{9}r$. Thus in all cases $|T_i| \leq \tfrac{7}{9}r$.\newline Finally, for each $v \in S_i$, recall $f(v)$ is the vertex $q_i-1$ vertices clockwise from $v$ and let $f(S_i) := \{f(v):v \in S\}$. Since $\mathcal{P}_i = \emptyset$ we have $S_i \cap f(S_i) = \emptyset$. Since $f$ is 1-1 it follows that $r \geq |S_i \cup f(S_i)| = |S_i|+|f(S_i)|=2|S_i|$, and thus $|S_i| \leq \tfrac{1}{2}r$.
\end{itemize}
%
Thus
\begin{align}
\tfrac{2}{3}(k-1) - \tfrac{1}{3} < \tfrac{1}{r}|S| &= \tfrac{1}{r}\sum_{i=1}^{k-1} |S_i| \leq \theta_1 + \theta_2 + \tfrac{1}{3}\theta_3 + \tfrac{1}{2}\theta_4 \text{, and} \label{eq:Sbound}\\
\tfrac{2}{3}(k-1) - \tfrac{1}{3} < \tfrac{1}{r}|T| &= \tfrac{1}{r}\sum_{i=1}^{k-1} |T_i| \leq \tfrac{1}{3}\theta_1 + \theta_3 + \tfrac{7}{9}\theta_4. \label{eq:Tbound}
\end{align}

Inequalities \eqref{eq:Sbound} and \eqref{eq:Tbound}, together with $k-1 = \theta_1 + \theta_2 + \theta_3 + \theta_4$ and $\theta_1 \leq 1$ form a set of integer linear inequalities. If $k-1 \geq 6$, then it is easily shown by hand or by computer that this set of inequalities has no feasible solution\footnotemark[2]\footnotetext[2]{This was computed using AMPL and Couenne 0.4.3, an Open-Source solver for Mixed Integer Nonlinear Optimization, available at \url{http://www.neos-server.org}.}. 

Finally, assume that $k-1=5$. In this case, by the above upper bounds on $|S_i|$ and $|T_i|$, we may replace \eqref{eq:Sbound} and \eqref{eq:Tbound} with the following more precise bounds: 
\begin{align*}
\tfrac{2}{3}(5)r - \tfrac{1}{3}r < |S| &\leq r\theta_1 + r\theta_2 + \floor{\tfrac{1}{3}r}\theta_3 + \floor{\tfrac{1}{2}r}\theta_4 \text{,} \qquad\qquad \!\text{and}\\
\tfrac{2}{3}(5)r - \tfrac{1}{3}r < |T| &\leq (\ceil{\tfrac{1}{3}r}-1)\theta_1 + r\theta_3 + (2\floor{\tfrac{1}{3}r}+1)\theta_4.
\end{align*}
Let $r=3a+b$ where $a \geq 1$ and $b \in \{0,1,2\}$. That is,
\begin{align}
9a+3b+1 &\leq (3a+b)(\theta_1 + \theta_2) + a\theta_3 + \tfrac{1}{2}(3a+b)\theta_4 \label{eq:Sstrong}\\
9a+3b+1 &\leq (\ceil{a+\tfrac{1}{3}b}-1)\theta_1 + (3a+b)\theta_3 + (2a+1)\theta_4. \label{eq:Tstrong}
\end{align}
%

Inequalities \eqref{eq:Sstrong} and \eqref{eq:Tstrong}, together with $\theta_1 \in \{0,1\}$ and $b \in \{0,1,2\}$ and $\theta_1 + \theta_2 + \theta_3 + \theta_4 = 5$ form a set of integer quadratic inequalities. This set has no solution since there are finitely many choices for $\theta_1,\theta_2,\theta_3,\theta_4$ and $b$, and for each such choice, \eqref{eq:Sstrong} and \eqref{eq:Tstrong} reduce to linear inequalities in $a$, which have no solution\footnotemark[2].
%

Note that if $k-1 \in \{1,2,3,4\}$ then there are choices of $\theta_1,\theta_2,\theta_3,\theta_4$ and $r$ that satisfy the integer linear inequalities implied by our bounds on $|S_i|$ and $|T_i|$.
\end{proof}

\begin{lemma}
\label{lemma:lemmafordouble}
Let $C$ be a cycle and $S,T \subseteq V(C)$ such that $|S|,|T| > \tfrac{1}{3}|C|$. Then there exists a path $P$ with one end vertex in $S$ and the other in $T$ such that $2 \leq |P| \leq \tfrac{1}{6}|C|+4$.
\end{lemma}
\begin{proof}
Suppose there exist vertices $v,w \in S$ such that there are at least $\tfrac{1}{3}|C|+2$ vertices between $v$ and $w$ (non-inclusive) in both directions around $C$. If there exists a vertex $u \in T$ within the $\ceil{\tfrac{1}{6}|C|}+1$ vertices counter-clockwise from $v$, then the path from $u$ clockwise to $v$ has at least two vertices and at most $\ceil{\tfrac{1}{6}|C|}+2 \leq \tfrac{1}{6}|C|+3$ vertices, as required. A similar result holds for the $\ceil{\tfrac{1}{6}|C|}+1$ vertices clockwise from $v$, the $\ceil{\tfrac{1}{6}|C|}+1$ vertices counter-clockwise from $w$, and the $\ceil{\tfrac{1}{6}|C|}+1$ vertices clockwise from $w$. Label these sets $Z_1,Z_2,Z_3,Z_4$ respectively. If $Z_1 \cap Z_4 = \emptyset$ then $|Z_1 \cup Z_4| = 2(\ceil{\tfrac{1}{6}|C|}+1) \geq \tfrac{1}{3}|C|+2$, otherwise $Z_1 \cup Z_4$ covers all vertices from $v$ counter-clockwise to $w$ and so $|Z_1 \cup Z_4| \geq \tfrac{1}{3}|C|+2$. A similar result holds for $Z_2 \cup Z_3$, and hence $|Z_1 \cup Z_2 \cup Z_3 \cup Z_4| \geq \tfrac{2}{3}|C|+4$. Given that $|T| > \tfrac{1}{3}|C|$ it follows $(Z_1 \cup Z_2 \cup Z_3 \cup Z_4) \cap T \neq \emptyset$ and our desired path can be constructed.

Hence we may assume that any two vertices in $S$ have less than $\tfrac{1}{3}|C|+2$ vertices of $C$ between them. Suppose $u \in T$ is an interior vertex on the short path $D$ between two vertices $v,w \in S$. Let $A \subset D$ be the path from $v$ to $u$, and let $B \subset D$ be the path from $u$ to $w$. Clearly $|A|,|B| \geq 2$. If $|A|,|B| > \tfrac{1}{6}|C|+4$, then $\tfrac{1}{3}|C| + 4 > |D| = |A| + |B| - 1 > \tfrac{1}{3}|C| + 7$, which is a contradiction. Hence either $A$ or $B$ satisfies our lemma, and as such we may suppose no such $u$ exists. 

Fix a vertex $v \in S$. Let $L \subseteq S$ be the set of vertices $w$ where a short path from $v$ to $w$ travels counter-clockwise from $v$; let $R$ be the equivalent clockwise set. If $L = \emptyset$ and $R=\emptyset$, then $S = \{v\}$, but $|S| > \tfrac{1}{3}|C| \geq 1$, which is a contradiction. Thus, without loss of generality, $L \neq \emptyset$. Let $x \in L$ be the vertex of greatest counter-clockwise distance from $v$, and define $y$ to be the vertex of greatest clockwise distance from $v$ in $R$ (if $R \neq \emptyset$) or let $y:=v$ if $R=\emptyset$. Let $A$ be the path from $x$ clockwise to $v$ and $B$ the path from $v$ clockwise to $y$. The interior vertices of $A$ and $B$ are not in $T$ as they are both short paths. Let $Q:= A \cup B$, and note $Q$ contains all vertices of $S$. Construct $Q'$ from $Q$ by adding the $\floor{\tfrac{1}{6}|C|}+3$ vertices counter-clockwise from $x$ and the $\floor{\tfrac{1}{6}|C|}+3$ vertices clockwise from $y$; none of these added vertices may be in $T$ without creating a satisfactory path. Thus $(Q' - \{x,y,v\}) \cap T = \emptyset$ and either $|Q'| > \tfrac{1}{3}|C| + 2(\floor{\tfrac{1}{6}|C|}+3) \geq \tfrac{1}{3}|C| + 2(\tfrac{1}{6}|C|+2) = \tfrac{2}{3}|C|+4$ (if these added vertices do not overlap) or $Q' = C$ (if they do). In the first case, $|Q' - \{x,y,v\}| > \tfrac{2}{3}|C|+1$ and as such $Q' - \{x,y,v\}$ intersects $T$, thus creating a satisfactory path. In the second case, $T \subseteq \{x,y,v\}$. Given $|T| > \tfrac{1}{3}|C|$, it follows $|C| \leq 8$.

Given that $|S|,|T| \geq 2$, choose a vertex $a \in S$ and $b \in T-\{a\}$. Since $|C| \leq 8$, there is a path from $a$ to $b$ containing at most $\tfrac{|C|+2}{2} = \tfrac{1}{6}|C| + 4 + \tfrac{1}{3}|C|-3 < \tfrac{1}{6}|C|+4$ vertices, as required.
%
\end{proof}

\section{Open Problems}
We conclude by noting some possible extensions of our results. We mentioned the following conjecture previously.
\begin{conjecture}[\citet{superfast5}]
\label{conjecture:reedwood}
Every graph with average degree at least $\tfrac{4}{3}t-2$ contains every $2$-regular $t$-vertex graph as a minor.
\end{conjecture}
Conjecture~\ref{conjecture:reedwood} essentially extends Theorem~\ref{theorem:cycles} by removing the requirement that all cycles in the desired minor have the same order.

Recall that the original result of \citet{corrhaj} was a stronger result in terms of minimum degree (rather than average degree). The following conjecture would imply our result, and be a more direct generalisation of the \citeauthor{corrhaj} Theorem.
\begin{conjecture}
Every graph with minimum degree at least $\tfrac{2}{3}kr$ and at least $kr$ vertices contains $k$ disjoint cycles, each containing at least $r$ vertices.
\end{conjecture}

\citet{hwang1} recently made the following two conjectures.
\begin{conjecture}[\citet{hwang1}]
\label{conjecture:oddwang}
For every integer $k \geq 2$ and odd integer $r \geq 3$, every graph $G$ with at least $rk$ vertices and minimum degree at least $\tfrac{r+1}{2}k$ contains $k$ disjoint cycles, each containing at least $r$ vertices.
\end{conjecture}
\begin{conjecture}[\citet{hwang1}]
\label{conjecture:evenwang}
For every integer $k \geq 3$ and even integer $r \geq 6$, every graph $G$ with at least $rk$ vertices and minimum degree at least $\tfrac{r}{2}k$ contains $k$ disjoint cycles, each containing at least $r$ vertices, unless $k$ is odd and $rk+1 \leq |V(G)| \leq rk+r-2$.
\end{conjecture}

These conjectures follow from the best known lower bounds on the required minimum degree. Consider the split graph $\mathcal{S}_{k,r,n}$ with a clique of order $\ceil{\tfrac{r}{2}}k-1$, an $n$-vertex independent set, and all edges between the two. Such a graph does not contain $k$ disjoint cycles of order at least $r$ since any such cycle must contain at least $\ceil{\tfrac{r}{2}}$ vertices of the clique. However, the minimum degree of $\mathcal{S}_{k,r,n}$ is $\ceil{\tfrac{r}{2}}k-1$ when $n \geq 1$. Due to this example the lower bounds on the minimum degree in Conjectures~\ref{conjecture:oddwang} and \ref{conjecture:evenwang} are necessary.

As \citeauthor{hwang1} shows, the exception when $k$ is odd and $rk+1 \leq |G| \leq rk+r-2$ is due to the disjoint union $K_{(k-1)p + p + i} \cup K_{(k-1)p + p + j}$, where $0 \leq i \leq j \leq \tfrac{1}{2}r-1$ and $j \neq 0$ and $p = \tfrac{1}{2}r$. Due to the limited number of vertices, each component contains at most $\tfrac{k-1}{2}$ disjoint cycles of order at least $r$, and hence the graph does not contain $k$ such disjoint cycles in total. However, if $i \geq 1$ then the minimum degree is at least $\tfrac{k-1}{2}r + \tfrac{1}{2}r + i -1 \geq \tfrac{r}{2}k$. If $i=0$, we can achieve the same minimum degree by adding edges from one vertex of the second clique to all vertices of the first---this does not increase the number of disjoint cycles.

We conjecture the following weakening of Conjectures~\ref{conjecture:oddwang} and \ref{conjecture:evenwang} by considering the average degree instead of the minimum degree. Note that while the average degree of $\mathcal{S}_{k,r,n}$ tends to $2\ceil{\tfrac{r}{2}}k-2$ as $n \rightarrow \infty$, the average degree of the disjoint union of complete graphs (or its modification) is small enough that we no longer need to consider it. Thus, the lower bound follows only from the existence of $\mathcal{S}_{k,r,n}$.

\begin{conjecture}
\label{conjecture:myodd}
For every integer $k \geq 2$ and odd integer $r \geq 3$, every graph $G$ with average degree at least $(r+1)k-2$ contains $k$ disjoint cycles, each containing at least $r$ vertices.
\end{conjecture}
\begin{conjecture}
\label{conjecture:myeven}
For every integer $k \geq 3$ and even integer $r \geq 6$, every graph $G$ with average degree at least $rk-2$ and at least $rk$ vertices contains $k$ disjoint cycles, each containing at least $r$ vertices.
\end{conjecture}

Conjecture~\ref{conjecture:oddwang} or \ref{conjecture:evenwang} would imply Conjecture~\ref{conjecture:myodd} or \ref{conjecture:myeven} respectively. The fact that considering the average degree rather than the minimum degree eliminate a class of possible counterexamples is promising, and suggests that results of this kind may be simpler and cleaner.

Finally, we present the following extension of Conjectures~\ref{conjecture:myodd} and \ref{conjecture:myeven} to $2$-regular $t$-vertex graphs. This is also a strengthening of Conjecture~\ref{conjecture:reedwood}.

\begin{conjecture}
\label{conjecture:strong}
Let $H$ be $2$-regular $t$-vertex graph with $c$ odd order components, such that $H$ is not a $t$-vertex even cycle or $\tfrac{t}{4}$ cycles of order 4. Every graph with average degree at least $t+c-2$ and at least $t$ vertices contains $H$ as a minor. 
\end{conjecture}

We require that $H$ not be a single even cycle since the graph consisting of $n \geq 2$ disjoint copies of $K_{t-1}$ has average degree $t-2$ and at least $t$ vertices but contains no $t$-vertex cycle. 
With regards to the other exception in Conjecture~\ref{conjecture:strong}, consider the graph $\mathcal{S}_{k,r,n}'$ constructed from $\mathcal{S}_{k,r,n}$ by adding a perfect matching on the independent set of order $n$. The graph $\mathcal{S}_{k,r,n}'$ does not contain $k$ disjoint cycles of order at least four since any such cycle still requires two vertices of the clique. However, the average degree of $\mathcal{S}'$ tends towards $4k-1=t-1 > t-2$ as $n \rightarrow \infty$. The graph $\mathcal{S}_{k,r,n}'$ is one of the three exceptions presented in \citep{hwang1}; again, promisingly, the other two exceptions presented there are not exceptions when considering the average degree. Hence Conjecture~\ref{conjecture:strong} seems plausible, if strong.

%
\bibliographystyle{abbrvnat}
\bibliography{sparse,fastbib}

\end{document}